\documentclass[review]{elsarticle}

\usepackage{lineno,hyperref}
\modulolinenumbers[5]

\journal{Journal of \LaTeX\ Templates}

\usepackage{amssymb,amsmath,amsthm, latexsym}

\theoremstyle{plain}
\newtheorem{thm}{{\sc Theorem}}[section]
\newtheorem{lem}[thm]{\sc Lemma}

\newtheorem{cor}[thm]{\sc Corollary}

\newtheorem{defn}[thm]{\sc Definition}
\newtheorem{rmk}[thm]{\sc Remark}

\newcounter{num}
\renewcommand{\bar}{\overline}

\usepackage{amsmath,mathdots}

\begin{document}

\begin{frontmatter}

\title{Inverse Relations in Shapiro's Open Questions\tnoteref{mytitlenote}}
\tnotetext[mytitlenote]{Research supported by Daegu University Research Grant 2013}


\author[daegu]{Ik-Pyo Kim\corref{mycorrespondingauthor}}
\ead{kimikpyo@daegu.ac.kr}
\cortext[mycorrespondingauthor]{Corresponding author}
\address[daegu]{Department of Mathematics Education,
     Daegu University, Gyeongbuk, 38453, Republic of Korea}


\author[washington]{Michael J. Tsatsomeros}
\address[washington]{ Department of Mathematics and Statistics,
     Washington State University, Pullman, WA 99164, USA}
\ead{tsat@math.wsu.edu}

\begin{abstract}
As an inverse relation, involution with an invariant sequence plays a key role in combinatorics and features prominently in some of Shapiro's open questions [L.W. Shapiro, Some open questions about random walks, involutions, limiting distributions and generating functions, Adv. Appl. Math. 27 (2001) 585-596]. In this paper, invariant sequences are used to provide answers to some of these questions about the Fibonacci matrix and Riordan involutions.
\end{abstract}

\end{frontmatter}

\linenumbers

\section{Introduction}
\par
\setcounter{num}{1}
\setcounter{equation}{0}
Inverse relations play a pivotal role in many research topics in combinatorics \cite{Riordan}. Among the open questions posed by Shapiro \cite{Shapiroq},
{\bf Q2}, {\bf Q8}, and {\bf Q8.1} regard involutions as a trait of inverse relations and have been the research focus of several authors \cite{Cameron,Cheon1,Cheon2,Shapiroq}. The concept of a Riordan matrix and generalizations of the Pascal, Catalan, and Motzkin triangles \cite{Sprugnoli}
allow us in this paper to obtain in-depth answers to Shapiro's open questions above; our answers are naturally related to the Fibonacci matrix
and Riordan involutions.
\begin{defn}
	{\rm
\label{Riordan}
An infinite lower triangular matrix $R=(g(x), f(x))$ is a {\em Riordan matrix} provided that the
generating function of the $i$th column of $R$ is $g(x)f(x)^i$ for $i=0, 1, 2, \ldots$, where
$g(x)=g_0+g_1x+g_2x^2+\cdots$ and $f(x)=f_1x+f_2x^2+\cdots$ with $g_0 \neq 0$ and $f_1 \neq 0$.
Moreover, a Riordan matrix $R$ is called a {\em Riordan involution}
if $R^2=(1,x)$, and a {\em Riordan pseudo involution} if $(RD)^2=(1,x)$,
where $D={\rm diag}(1, -1, 1, -1, \ldots)$ \cite{Cheon2}. }
\end{defn}
The definition of Riordan (pseudo) involutions above is indeed facilitated by the fact that
the set of all Riordan matrices is a group under matrix multiplication, referred to as
the Riordan group \cite{Shapiro}, where multiplication amounts to
\begin{equation}
\label{b1}
(g(x),f(x))\;(h(x),l(x))\,=\,(g(x)h(f(x),l(f(x)).
\end{equation}

Shapiro's open questions mentioned above have long been of great interest in the investigation of
Riordan (pseudo) involutions, as well as the Riordan group; we restate them below, recalling that
$D={\rm diag}(1, -1, 1, -1, \ldots)$.

\begin{itemize}
  \item {\bf Q2}: Let
  \begin{equation*}
      {\mathbb F}=\left[\begin{array}{ccccccccc}
           ~ 1 & 0 & 0 & 0 & 0   &\cdots~\\
            ~0 & 1 & 0 & 0 & 0   &\cdots~\\
            ~0 & 1 & 1 & 0 & 0   &\cdots~\\
            ~0 & 0 & 2 & 1 & 0  & \cdots~\\
            ~0 & 0 & 1 & 3 & 1  &  \cdots~\\
             ~\vdots & \vdots & \vdots & \vdots &\vdots &  \ddots~
                   \end{array}\right]
\end{equation*}
denote the Fibonacci matrix \cite{Shapiroq}. Is there a combinatorial connection between ${\mathbb F}$ and $D{\mathbb F}^{-1}D$?
  \item {\bf Q8}  : Can every Riordan involution $R$ be written as $R=BDB^{-1}$ for some element $B$ in the Riordan group?
  \item {\bf Q8.1}: If a Riordan involution $R$ is of some particular combinatorial significance, can we find a $B$ in the
                    Riordan group, which has a related combinatorial significance and $R=BDB^{-1}$?
\end{itemize}


In \cite{Shapirosub}, several subgroups of the Riordan group are introduced. For example,
$\{(g(x), cx)|c \neq 0 \}$ and $\{(g(x), cxg(x))| c \neq 0\}$
are called the {\em $c$-Appell} and {\em $c$-Bell subgroups}, respectively.
In the case of $c=1$, these are simply called the {\em Appell} and {\em Bell subgroups}, respectively.
In \cite{Cheon1}, Cheon and Kim showed the existence of Riordan matrices $B$ as affirmative answers to {\bf Q8} and {\bf Q8.1}
by using an antisymmetric function and by adopting the Bell subgroup, respectively. In \cite{Cheon2}, Cheon et al. presented a pseudo involution
$R_n~(n=0, 1, \ldots)$ as a generalization of the RNA triangle such that for any generating function $G(x)$ with $G(0)\neq 0$
and for each nonnegative integer $n$,
$R_n=\left( g(x)\left( {{G(x)} \over {G(1-xg(x))}}\right)^n, xg(x)\right).$
The latter equality is a revised form of Cameron and Nkwanta's example in \cite{Cameron}, that is,
$W_n=\left( g(x)\left( {{1-x} \over {1-xg(x)}}\right)^n, xg(x)\right) \quad\mbox{with}\quad
       W_n=A^{-n}W_0A^n;$
the latter can be thought of as a partial answer to {\bf Q8},
where $g(x)={{(1-x+x^2)-\sqrt{(1-x+x^2)^2-4x^2}} \over 2x^2}$ and $A^n=\left({1 \over (1-x)^n}, x\right)$.

The shared notion of involutions as self-inverse relations in {\bf Q8} and {\bf Q8.1}, as well as the Fibonacci
and Catalan numbers, have been extensively studied. Little attention has been paid, however, to {\bf Q2}. In fact,
the row sums of ${\mathbb F}$ and $D{\mathbb F}^{-1}D$ are the Fibonacci and Catalan numbers, respectively. Among the objectives of our
research is to present plausible answers to {\bf Q2}, and also to give answers related to invariant sequences as self-inverse relations
\cite{Kim,Riordan,Wang} to {\bf Q8} and {\bf Q8.1}.
More specifically, in this paper, we investigate the structure of entries in $D{\mathbb F}^{-1}D$ and the role of $D{\mathbb F}^{-1}D$
in transforming invariant sequences, giving rise to answers for {\bf Q2}.
We also provide a method for constructing invariant sequences \cite{Kim} by means of Riordan (pseudo) involutions,
which allows us to answer both {\bf Q8} and {\bf Q8.1}.
\section{Notation and preliminaries}
\par
\setcounter{num}{2}
\setcounter{equation}{0}
We begin with a slight extension of the notion of Riordan matrix in Definition \ref{Riordan}.
\begin{defn} {\rm
\label{semi-Riordan}
An infinite lower triangular matrix $R$ is called a {\em semi-Riordan matrix} if the generating function of the $i$th column of $R$ is $g(x)f(x)^i$ for $i=0, 1, \ldots$,
where $g(x)=g_0+g_1x+g_2x^2+g_3x^3+\cdots$ and $f(x)=f_1x+f_2x^2+f_3x^3+\cdots$. For convenience, we will denote it by $R=(g(x),f(x))$
similarly to the Riordan matrix abbreviation, the distinction being clear from the context. }
\end{defn}
From now on, infinite real sequences $\{x_n\}$ are identified with the infinite dimensional real vector space $\mathbb{R}^\infty$ consisting of column vectors ${\mathbf x}=[x_0, x_1, \ldots]^T$.
For two semi-Riordan matrices $(g(x), f(x))$ and $(h(x), l(x))$, matrix multiplication is defined
as in the Riordan group. Matrix addition can be defined only when
$f(x)=l(x)$ as $(g(x), f(x))+((h(x), f(x))=(g(x)+h(x), f(x))$, because for all $u(x)$,
\begin{equation}
\label{b2}
\begin{split}
(g(x), f(x))u(x)+((h(x), f(x))u(x)&=g(x)u(f(x))+h(x)u(f(x))\\
                                  &=(g(x)+h(x))u(f(x))\\
                                   &=(g(x)+h(x), f(x))u(x),
\end{split}
\end{equation}
where $u(x)$ is the generating function of a column vector in $\mathbb{R}^\infty$.

We let $\mathbb{E}_\lambda (A)$ denote the eigenspace of a (finite or infinite)
matrix $A$ corresponding to its eigenvalue $\lambda$.
We generalize the notions of invariant and inverse invariant sequences of the first or second kind in \cite{Kim} as follows:
\begin{defn} {\rm
\label{Rinvseq}
For a Riordan involution (resp., pseudo involution) $R$, $\mathbf{x}\in \mathbb{R}^\infty$ is called
\begin{itemize}
\item [(a)]
an  {\em $R$-invariant sequence of the first kind} if $\mathbf{x} \in \mathbb{E}_{1} (R)$ ~(resp., $\mathbf{x} \in \mathbb{E}_{1} (RD)$).
\item [(b)]
an {\em inverse $R$-invariant sequence of the first kind} if $\mathbf{x} \in \mathbb{E}_{-1} (R)$ ~(resp., $\mathbf{x} \in \mathbb{E}_{-1} (RD)$).
\item [(c)]
an {\em $R$-invariant sequence of the second kind} if $\mathbf{x} \in \mathbb{E}_{1} (R^T)$ ~(resp., $\mathbf{x} \in \mathbb{E}_{1} (R^TD)$).
\item [(d)]
an {\em inverse $R$-invariant sequence of the second kind if} $\mathbf{x} \in \mathbb{E}_{-1} (R)$ ~(resp., $\mathbf{x} \in \mathbb{E}_{-1} (R^TD)$).
\end{itemize} }
\end{defn}

Let $P=\left[\begin{matrix}
         {i}\choose{j}
         \end{matrix} \right]$
$(i,j=0, 1, 2, \ldots)$ denote the (infinite) {\em Pascal matrix}.
Let $\mathbf{F}=[F_0, F_1, F_2, \ldots ]^T$ and $\mathbf{L}=[L_0, L_1, L_2, \ldots]^T$
denote the vectors in $\mathbb{R}^{\infty}$ whose entries are the members of the Fibonacci and Lucas
sequences, respectively; that is
\[ F_0=0, \; F_1=1, \; F_n=F_{n-1}+F_{n-2} \;\; (n\geq 2), \]
\[ L_0=2, \; F_1=1, \; L_n=L_{n-1}+L_{n-2} \;\; (n\geq 2). \]
$\mathbf{F}$ and $\mathbf{L}$ are $P$-invariant and inverse $P$-invariant sequences of the first kind, respectively,
as a pseudo involution $P$ \cite{Choi, Kim}.

For a constant $a$, we let $J(a)$ denote the infinite Jordan block of the form
\[
    \left[ \begin{array}{ccccccc}
          a & 1   &          &         &      & \\
            & a   &  1       &         &     O& \\
              &     &  \ddots &   \ddots &  & \\
              &     &         &  a & 1 &\\
              &  O  &         &   & a  &\ddots\\
              &     &          &         &  & \ddots
         \end{array} \right].
\]
It can be easily proven that $J(0)\mathbf{F}$ and $J(0)\mathbf{L}$ are $P$-invariant and inverse
$P$-invariant sequences of the second kind, respectively; see \cite{Kim}.

For a matrix $A$ with columns $A_j$ ($j=0,1, 2, \ldots$) and with ${\mathbf 0}_j$ denoting the vector of zeros
in $\mathbb{R}^j$, we let ${A\hspace{0.01cm}^\downarrow}$ denote the matrix whose $j$th column is
${\displaystyle
\left[\begin{matrix}
\mathbf 0_j\\
A_j
\end{matrix} \right]}$ where $\mathbf 0_0$ is vacuous.
The following result leads us to investigate the relationships between the Fibonacci matrix ${\mathbb F}$ and $D{\mathbb F}^{-1}D$ \cite{Kim}; the last two
clauses appeared in \cite{Choi}. Note that, by utilizing (\ref{b1}) and (\ref{b2}), we have expressed the result in terms of the semi-Riordan matrices.
\begin{lem}
\label{columns}
Let $P=\left({1 \over {1-x}}, {x \over {1-x}}\right)$, $D=(1, -x)$ and $Q=\left({{2-x} \over {1-x}}, {x \over {1-x}}\right)$. Then the following hold:
\begin{itemize}
  \item [(a)] The columns of $P^{T\downarrow}=(1, x(1+x))$ form a basis for $\mathbb{E}_1(P^TD)$.
  \item [(b)] The columns of $Q^{T\downarrow}(0|0)=(1+2x,x(1+x))$ form a basis for $\mathbb{E}_{-1}(P^TD)$.
  \item [(c)] The columns of
$\left[\begin{array}{c}
{\mathbf 0}^T\\
{P\hspace{0.01cm}^\downarrow}
\end{array} \right]=\left({x \over {1-x}}, {x^2 \over {1-x}}\right)$
form a basis for $\mathbb{E}_{-1}(PD)$.
  \item [(d)] The columns of
${Q\hspace{0.01cm}^\downarrow}=\left({{2-x} \over {1-x}}, {x^2 \over {1-x}}\right)$ form a basis for $\mathbb{E}_1(PD)$.
\end{itemize}
\end{lem}

Surprisingly, $P^{T\downarrow}$ is equal to the Fibonacci matrix ${\mathbb F}$,
the columns of which form a basis for $\mathbb{E}_1(P^TD)$.
\begin{defn}
	{\rm
We call the four special matrices
$P^{T\downarrow}$, $Q^{T\downarrow}(0|0)$, $\left[\begin{array}{c}
{\mathbf 0}^T\\
{P\hspace{0.01cm}^\downarrow}
\end{array} \right]$, and ${Q\hspace{0.01cm}^\downarrow}$
the {\em Fibonacci matrix of the second kind}, the {\em Lucas matrix of the second kind}, the
{\em Fibonacci matrix of the first kind}, and the {\em Lucas matrix of the first kind}, and we will denote them by
${\mathbb F}^{\rm S}, {\mathbb L}^{\rm S}, {\mathbb F}^{\rm F}$, and ${\mathbb L}^{\rm F}$, respectively.}
\end{defn}

For each Riordan matrix $R=(g(x), f(x))$, we consider its inverse matrix given by
\begin{equation}
\label{b3}
R^{-1}=(1/g(\bar f(x)), \bar f(x)),
\end{equation}
where $\bar f(x)$ is the compositional inverse of $f(x)$,
i.e., $f(\bar f(x))=\bar f(f(x))=x$ \cite{Shapiro}.

For ${\mathbb F^{\rm S}}=(1, x(1+x))$ resp. ${\mathbb L^{\rm S}}=(1+2x, x(1+x))$, we can obtain  by (\ref{b3}) the well-known result that
${(\mathbb F^{\rm S})}^{-1}=(1, {{-1+ \sqrt{1+4x}}\over 2})$ resp. ${(\mathbb L^{\rm S})}^{-1}=({1 \over \sqrt{1+4x}}, {{-1+ \sqrt{1+4x}}\over 2})$.
Therefore,
\[
    {(\mathbb F^{\rm S})}^{-1}=\left[\begin{array}{ccccccccc}
           ~ 1 & 0 & 0 & 0 & 0  &   \cdots~\\
            ~0 & 1 & 0 & 0 & 0  &   \cdots~\\
            ~0 & -1 & 1 & 0 & 0 &   \cdots~\\
            ~0 & 2 & -2 & 1 & 0  &   \cdots~\\
            ~0 & -5 & 5 & -3 & 1 &   \cdots~\\
         ~\vdots & \vdots & \vdots & \vdots &\vdots & \ddots~
                   \end{array}\right],~
      {(\mathbb L}^{\rm S})^{-1}=\left[\begin{array}{ccccccccc}
           ~ 1 & 0 & 0 & 0 & 0  &   \cdots~\\
            ~-2 & 1 & 0 & 0 & 0  &   \cdots~\\
            ~6 & -3 & 1 & 0 & 0  &   \cdots~\\
            ~-20 & 10 & -4 & 1 & 0 &   \cdots~\\
            ~70 & -35 & 15 & -5 & 1 &   \cdots~\\
                ~\vdots & \vdots & \vdots & \vdots &\vdots & \ddots~
                   \end{array}\right].
\]
The relationships between these four special matrices and $D({\mathbb L}^{\rm S})^{-1}D$, as well as $D({\mathbb F}^{\rm S})^{-1}D$,
play a key role in our investigation.

Our effort in this paper will proceed as follows:

In section $3$, we show that
\begin{itemize}
  \item $D({{\mathbb F}^{\rm S}})^{-1}D=\left(1, xC(x)\right)=P\left({{M(x)-xM(x)-x^2M(x)^2} \over {1+x}}, {{x+x^2M(x)} \over {1+x}}\right)$,
  \item $D({{\mathbb L}^{\rm S}})^{-1}D=\left({1\over {1-2xC(x)}}, xC(x)\right)=\left(W(x), x\right)D({{\mathbb F}^{\rm S}})^{-1}D$
  \end{itemize}
by utilizing the Catalan and Motzkin numbers \cite{Donaghey}, where $C(x)={{1- \sqrt{1-4x}}\over 2x}$, $M(x)={{1-x- \sqrt{(1-x)^2-4x^2}}\over 2x^2}$, and $W(x)=\sum_{n=0}^{\infty}{\binom{2n}{n}}x^n$. In the process, we show that each entry of $D({\mathbb F}^{\rm S})^{-1}D$ can be expressed as a linear combination of Catalan numbers with coefficients connected to the Fibonacci sequence; and each entry of $D({{\mathbb L}^{\rm S}})^{-1}D$ can be expressed as a linear combination
of Catalan numbers with coefficients connected to the Fibonacci and Lucas sequences. This amounts to an answer to {\bf Q2}.
In section $4$, we show that $D(\mathbb F^{\rm S})^{-1}D$ and $D(\mathbb F^{\rm S})^{-1}D$ provide a mechanism (featuring the Catalan numbers) for transforming
$P$-invariant into an inverse $P$-invariant sequence of the second kind, and vice versa. This provides an answer to {\bf Q2}, as a combinatorial relationship between $\mathbb F^{\rm S}$ and $D(\mathbb F^{\rm S})^{-1}D$. We also present the relationships (featuring the Catalan and Motzkin numbers) between the generating functions of $P$-invariant and inverse $P$-invariant sequences of the second kind.
In section $5$, we provide a method for constructing $R$-invariant sequences of the first or second kind by means of the Riordan (pseudo) involution $R$ itself.
It follows that for every Riordan involution $R$ in the $(-1)$-Appell subgroup, there exist $B^n \,(n=1, 2, \ldots)$ in the Appel subgroup
such that $R=B^nDB^{-n}$, thus providing answers to {\bf Q8} and {\bf Q8.1}.

\section{The structure of $D({\mathbb F}^{\rm S})^{-1}D$ and $D({\mathbb L}^{\rm S})^{-1}D$}
\par
\setcounter{num}{3}
\setcounter{equation}{0}
In this section, we investigate the structure of the matrices $D({\mathbb F}^{\rm S})^{-1}D$ and $D({\mathbb L}^{\rm S})^{-1}D$, using the $n$th Catalan number
$C_n={1 \over n+1} \binom{2n}{n}$ and the Motzkin number $M_n=\sum_{k=0}^{\lfloor {n\over 2}\rfloor}{\binom{n}{2k}}C_k$ for $n=0, 1, 2, \ldots$ \cite{Donaghey}.
This contributes to the investigation of the structures of $D({\mathbb F}^{\rm S})^{-1}D$ and $D({\mathbb L}^{\rm S})^{-1}D$, leading to answers for {\bf Q2}.
From now on, $C(x)$ and $M(x)$ denote the generating functions of the Catalan numbers and Motzkin numbers, respectively.
\begin{lem}
\label{inverse}
Let ${\mathbb F}^{\rm S}$ and ${\mathbb L}^{\rm S}$ denote the Fibonacci and Lucas matrices of the second kind, respectively.
Let $P$ be the Pascal matrix. Then the following hold:
\begin{itemize}
  \item [(a)] $D({\mathbb F}^{\rm S})^{-1}D=\left(1, xC(x)\right)=P\left({{M(x)-xM(x)-x^2M(x)^2} \over {1+x}}, {{x+x^2M(x)} \over {1+x}}\right)$,
  \item [(b)] $D({\mathbb L}^{\rm S})^{-1}D=\left({1\over {1-2xC(x)}}, xC(x)\right)=\left(W(x), x\right)P\left({{M(x)-xM(x)-x^2M(x)^2} \over {1+x}}, {{x+x^2M(x)} \over {1+x}}\right)$,
  \item [(c)] $D({\mathbb L}^{\rm S})^{-1}D=\left({1\over {1-2xC(x)}}, x\right)D({\mathbb F}^{\rm S})^{-1}D=\left(W(x), x\right)D({\mathbb F}^{\rm S})^{-1}D=D({\mathbb F}^{\rm S})^{-1}D\left({1\over {1-2x}}, x\right)$,
\end{itemize}
where $C(x)={{1- \sqrt{1-4x}}\over 2x}$, $M(x)={{1-x- \sqrt{(1-x)^2-4x^2}}\over 2x^2}$, and $W(x)=\sum_{n=0}^{\infty}{\binom{2n}{n}}x^n$.
\end{lem}
\begin{proof}
It directly follows from (\ref{b1}) and (\ref{b3}) that $D({\mathbb F}^{\rm S})^{-1}D=\left(1, {{1- \sqrt{1-4x}}\over 2}\right)$ and
$D({\mathbb L}^{\rm S})^{-1}D=\left({1\over {\sqrt{1-4x}}},{{1- \sqrt{1-4x}}\over 2}\right)$. Since $T({{C(x)-1} \over x})=M(x)$, which is the result in \cite{Donaghey}, we obtain
\begin{equation}
\label{CMotzkin}
\left({1 \over {1+x}}, {x \over {1+x}}\right)\left({{C(x)-1} \over x}, x\right)=\left(M(x), {x \over {1+x}}\right),
\end{equation}
where $T(f(x))={1 \over {1+x}}f({x\over {1+x}})$ is the Euler transformation and
$$M(x)={{1-x- \sqrt{(1-x)^2-4x^2}}\over 2x^2}.$$
By Newton's binomial theorem \cite{Brualdi,Comtet},
$(1-4x)^{-{1 \over 2}}=\sum_{n=0}^{\infty}{\binom{-{1 \over 2}}{n}}(-4x)^n=\sum_{n=0}^{\infty}{\binom{2n}{n}}x^n$
and ${{1- \sqrt{1-4x}}\over 2}=\sum_{n=0}^{\infty}{1 \over {n+1}}{\binom{2n}{n}}x^{n+1}$. Thus,
 $D({\mathbb F}^{\rm S})^{-1}D=\left(1, xC(x)\right)$ by $C(x)={{1- \sqrt{1-4x}}\over 2x}$ \cite{Shapiro}. Since $P\left({1 \over {1+x}}, {x \over {1+x}}\right)=(1,x)$ and $\left({{C(x)-1} \over x}, x\right)\left({x \over {C(x)-1}}, xC(x)\right)=(1, xC(x))$, by (\ref{CMotzkin}) we
 obtain
$$(1, xC(x))=P\left(M(x), {x \over {1+x}}\right)\left({x \over {C(x)-1}}, xC(x)\right),$$
which implies (a) from the fact that ${x \over {C(x)-1}}=1-x-xC(x)$ and $C({x\over {1+x}})=1+xM(x)$.

Clause (b) follows from (a) and $D({\mathbb L}^{\rm S})^{-1}D=(W(x), x)(1, xC(x))$ since $\sqrt{1-4x}=1-2xC(x)$.

Clause (c) is a consequence of  (a), (b), and $D({\mathbb L}^{\rm S})^{-1}D=(1, xC(x))\left({1\over {1-2x}}, x\right).$
\end{proof}
The following theorem contains the recurrence relations for the entries of
$D({\mathbb F}^{\rm S})^{-1}D$ and $D({\mathbb L}^{\rm S})^{-1}D$.
\begin{thm}
\label{structure}
Let $D({\mathbb F}^{\rm S})^{-1}D=[r_{ij}]$ and $D({\mathbb L}^{\rm S})^{-1}D=[q_{ij}]$ for $i$ and $j$ with $i, j=0, 1, 2, \ldots$. Then
\begin{itemize}
  \item [(a)] $r_{00}=1$ and for $i=1, 2, \ldots;~ j=2, 3, \ldots$,
\begin{equation}
\label{Fibore}
r_{i0}=0, \;\; r_{i1}={1\over i}{\binom{2i-2}{i-1}}, \;{\text and}~\; r_{ij}=-r_{i-1,j-2}+r_{i,j-1}.
\end{equation}

  \item [(b)] $q_{i0}={\binom{2i}{i}}$ for $i=0,1,2,\ldots$, $\;q_{i1}={1\over 2}{\binom{2i}{i}}$ for $i=1, 2, \ldots$, and
\begin{equation}
\label{Lucare}
q_{ij}=-q_{i-1,j-2}+q_{i,j-1}
\end{equation}
for $i=1, 2, \ldots;~ j=2, 3, \ldots$
\end{itemize}
\noindent
\end{thm}
\begin{proof}
Let $g(x)^{s-1}\left({{1- \sqrt{1-4x}}\over 2}\right)^j$ be the generating function of the $j$th column of $D({\mathbb F}^{\rm S})^{-1}D$ (when $s=1$, $g(x)=1$) or $D({\mathbb L}^{\rm S})^{-1}D$
(when $s=2$, $g(x)={1\over \sqrt{1-4x}}$). For $j\geq 2$,
\[
\begin{split}g(x)^{s-1}\left({{1- \sqrt{1-4x}}\over 2}\right)^j&=g(x)^{s-1}\left({{1- \sqrt{1-4x}}\over 2}\right)^{j-2}\left({{1- \sqrt{1-4x}}\over 2}-x\right)\\
&=g(x)^{s-1}\left({{1- \sqrt{1-4x}}\over 2}\right)^{j-1}- xg(x)^{s-1}\left({{1- \sqrt{1-4x}}\over 2}\right)^{j-2},
\end{split}
\]
which, along with Lemma \ref{inverse}, imply that (a) and (b) hold.
\end{proof}
To illustrate the above theorem, all the entries of $D({\mathbb F}^{\rm S})^{-1}D$ and $D({\mathbb L}^{\rm S})^{-1}D$
can be completely determined by the recurrence relations in Theorem \ref{structure}, once the entries in the first and second columns are determined;
indeed, we have
\[
      D({\mathbb F}^{\rm S})^{-1}D=\left[\begin{array}{ccccccccc}
           ~ 1 & 0 & 0 & 0 & 0  &   \cdots~\\
            ~0 & 1 & 0 & 0 & 0  &   \cdots~\\
            ~0 & 1 & 1 & 0 & 0  &   \cdots~\\
            ~0 & 2 & 2 & 1 & 0   &  \cdots~\\
            ~0 & 5 & 5 & 3 & 1  &   \cdots~\\
            ~\vdots & \vdots & \vdots & \vdots &\vdots &  \ddots~
                   \end{array}\right],~
      D({\mathbb L}^{\rm S})^{-1}D=\left[\begin{array}{ccccccccc}
           ~ 1 & 0 & 0 & 0 & 0  &   \cdots~\\
            ~2 & 1 & 0 & 0 & 0  &   \cdots~\\
            ~6 & 3 & 1 & 0 & 0  &   \cdots~\\
            ~20 & 10 & 4 & 1 & 0 &   \cdots~\\
            ~70 & 35 & 15 & 5 & 1 &   \cdots~\\
                ~\vdots & \vdots & \vdots & \vdots &\vdots & \ddots~
                   \end{array}\right].
\]

The following is a result about row sums of the matrices $D({\mathbb F}^{\rm S})^{-1}D$
and $D({\mathbb L}^{\rm S})^{-1}D$.
\begin{cor}
\label{rowsum}
Let $D({\mathbb F}^{\rm S})^{-1}D=[r_{ij}]$ and $D({\mathbb L}^{\rm S})^{-1}D=[q_{ij}]$ for $i$ and $j$ with $i, j=0, 1, 2,\ldots$.
Then for positive integers $i$ and $j$ with $i \geq j$, the following hold:
\begin{itemize}
  \item [(a)] $r_{i,i-j+1}=r_{i-1,i-1}+r_{i-1,i-2}+\cdots+r_{i-1,i-j}$.
  \item [(b)] $q_{i, i-j+1}=q_{i-1,i-1}+q_{i-1,i-2}+\cdots+q_{i-1,i-j}$.
  \end{itemize}
\end{cor}
\begin{proof}
Let $i$ and $j$ be positive integers with $i \geq j$.
We prove the result by induction on $j$ for $j \geq 1$. Since $r_{i, i+1}=0$ by the structure of
$D({\mathbb F}^{\rm S})^{-1}D$ and $r_{i, i+1}=-r_{i-1, i-1}+r_{i,i}$ by (\ref{Fibore}), we get that
$r_{i, i}=r_{i-1,i-1}$ for $j=1$. Let $j \geq 2$. Then by the induction hypothesis and (\ref{Fibore}) again,
\begin{equation}
\label{suminrows1}
r_{i,i-j+2}=r_{i-1,i-1}+r_{i-1,i-2}+\cdots+r_{i-1,i-j+1}
\end{equation}
and $r_{i, i-j+2}=-r_{i-1, i-j}+r_{i-1,i-j+1}$. Thus we get the result by replacing $r_{i, i-j+2}$ with
$-r_{i-1, i-j}+r_{i-1,i-j+1}$ in (\ref{suminrows1}).
The second part of the corollary can be proven similarly.
\end{proof}
Let ${\mathbf e}_i$  $(i=0, 1, 2, \ldots)$ denote the $i$th column of the identity matrix $I$,
and let ${\mathbf e}={\mathbf e}_0+{\mathbf e}_1+{\mathbf e}_2+\cdots$.
Then we directly obtain the following result from Corollary \ref{rowsum}.
\begin{cor}
\label{suminrows2}
Let ${\mathbb F}^{\rm S}$ and ${\mathbb L}^{\rm S}$ denote the Fibonacci and Lucas matrices of the second kind, respectively. Then the following hold:
\begin{itemize}
  \item [(a)] $D({\mathbb F}^{\rm S})^{-1}D {\mathbf e}=[C_0, C_1, C_2, \ldots, C_n, \ldots]^T$.
  \item [(b)] $D({\mathbb L}^{\rm S})^{-1}D {\mathbf e}=[C_0, 3C_1, 5C_2, \ldots, (2n+1)C_n, \ldots]^T$.
  \end{itemize}
\end{cor}
\begin{proof}
Let $D({\mathbb F}^{\rm S})^{-1}D=[r_{ij}]$ and $D({\mathbb L}^{\rm S})^{-1}D=[q_{ij}]$ for $i, j=0, 1,\ldots\,$
For positive integers $i$ and $j$ with $i \geq j$, it follows from the case of $j=i$ in Corollary \ref{rowsum} that
\begin{equation}
\begin{split}
r_{i1}=r_{i-1,0}+r_{i-1,1}+\cdots+r_{i-1,i-1},\\
q_{i1}=q_{i-1,0}+q_{i-1,1}+\cdots+q_{i-1,i-1},\\
\end{split}
\end{equation}
which imply that $r_{i1}$ and $q_{i1}$ are the $(i-1)$th row sums with $r_{i1}={1\over i}{\binom{2i-2}{i-1}}$ and $q_{i1}={1\over 2}{\binom{2i}{i}}$.
Thus for $i=0,1, 2, \ldots$, the $i$th row sums of $D({\mathbb F}^{\rm S})^{-1}D$ resp. $D({\mathbb L}^{\rm S})^{-1}D$ are ${1\over {i+1}}{\binom{2i}{i}}$ resp. ${(2i+1)\over (i+1)}{\binom{2i}{i}}$, and (a) and (b) is proven.
\end{proof}

We now present one of our main results as an answer to {\bf Q2}. It provides a combinatorial relationship between ${\mathbb F}^{\rm S}$ and
$D({\mathbb F}^{\rm S})^{-1}D$ and furthermore, between ${\mathbb L}^{\rm S}$ and $D({\mathbb L}^{\rm S})^{-1}D$.
\begin{thm}
\label{FiboLucaentry}
Let ${\mathbb F}^{\rm S}$ and ${\mathbb L}^{\rm S}$ denote the Fibonacci and Lucas matrices of the second kind, respectively. Then the following hold:
\begin{itemize}
  \item [(a)] $D({\mathbb F}^{\rm S})^{-1}D {\mathbf e}_0=[C_0, 0, 0, \ldots]^T$; if $j\geq 1$, then $D({\mathbb F}^{\rm S})^{-1}D {\mathbf e}_j=[r_{0j}, r_{1j}, \ldots, r_{nj}, \ldots]^T$
  where
 $$r_{nj}=\left\{
 \begin{array}{cc}
 {\binom{j-1}{0}}C_{n-1}-{\binom{j-2}{1}}C_{n-2}+\cdots+(-1)^{\lfloor{{{j-1}\over 2}\rfloor}}{\binom{\lceil{{j-1}\over 2}\rceil}{\lfloor{{j-1}\over 2}\rfloor}}C_{n-\lceil{{{j}\over 2}\rceil}}, & \mbox{\;\;if\;\;} n\geq j,\\
 0, & \mbox\;{otherwise.}
 \end{array}\right.
 $$

  \item [(b)] $D({\mathbb L}^{\rm S})^{-1}D {\mathbf e}_0=[C_0, 2C_1, 3C_2, \ldots]^T$;
  if $j\geq 1$, then $D({\mathbb L}^{\rm S})^{-1}D {\mathbf e}_j=[q_{0j}, q_{1j}, \ldots, q_{nj}, \ldots]^T$,
  where for $n \geq 1$, $q_{n1}=(2n-1)C_{n-1}$ and for $j\geq 2$ with $n \geq j$,
\begin{equation}
  \label{Lucarel}
\begin{split}
q_{nj}&=(n-1)C_{n-1}-\left[(n-2)\left({\binom{j-2}{1}}+{\binom{j-3}{0}}\right)+{\binom{j-3}{0}}\right]C_{n-2}\\
&+\left[(n-3)\left({\binom{j-3}{2}}+{\binom{j-4}{1}}\right)+{\binom{j-4}{1}}\right]C_{n-3}+\cdots\\
&+(-1)^{\lfloor{{{j-1}\over 2}\rfloor}}\left[\left(n-\lceil{{{j}\over 2}\rceil}\right)\left({\binom{\lceil{{j-1}\over 2}\rceil}{\lfloor{{j-1}\over 2}\rfloor}}+{\binom{\lceil{{j-1}\over 2}\rceil-1}{\lfloor{{j-1}\over 2}\rfloor-1}}\right)+{\binom{\lceil{{j-3}\over 2}\rceil}{\lfloor{{j-3}\over 2}\rfloor}}\right]C_{n-{\lceil{{{j}\over 2}\rceil}}}.
\end{split}
\end{equation}
\end{itemize}
\end{thm}
\begin{proof}
\noindent
(a) The proof proceeds by induction on $j$ with $j\geq 1$.
For $j=1$ and $2$, we know that the generating function for the $j$th column
of $D({\mathbb F}^{\rm S})^{-1}D$ is $xC(x)$ if $j=1$
and $x^2C(x)^2$ if $j=2$ (see Lemma \ref{inverse} (a)). Since $C_{n+1}=\sum_{k=0}^nC_kC_{n-k}~ (n \geq 0)$ \cite{Brualdi}, $$x^2C(x)^2=x^2\sum_{n=0}^{\infty}(C_0C_{n}+C_1C_{n-1}+\cdots+C_nC_0)x^n=x^2\sum_{n=0}^{\infty}C_{n+1}x^n=\sum_{n=2}^{\infty}C_{n-1}x^n.$$
Thus we have $r_{n1}=C_{n-1}$ for $n \geq 1$ and $r_{n2}=C_{n-1}$ for $n \geq 2$, and we can commence the induction. Let $j \geq 3$. If $j$ is odd,
then by the induction hypothesis, we have
\begin{equation}
\label{rowsum4}
\begin{split}
r_{n-1,j-2}&={\binom{j-3}{0}}C_{n-2}+\cdots+(-1)^{{j-3 \over 2}-1}{\binom{{j-1\over 2}}{{j-3\over 2}-1}}C_{n-{j-1\over 2}}+(-1)^{{j-3 \over 2}}{\binom{{j-3\over 2}}{{j-3\over 2}}}C_{n-1-{j-1\over 2}},\\
r_{n,j-1}&={\binom{j-2}{0}}C_{n-1}-{\binom{j-3}{1}}C_{n-2}+\cdots+(-1)^{{j-3 \over 2}}{\binom{{j-1\over 2}}{{j-3\over 2}}}C_{n-{j-1\over 2}}.
\end{split}
\end{equation}
By (\ref{Fibore}) applied to (\ref{rowsum4}),
\begin{equation}
\label{rowsum5}
r_{n,j}={\binom{j-1}{0}}C_{n-1}-{\binom{j-2}{1}}C_{n-2}+\cdots+(-1)^{{j-3 \over 2}}{\binom{{j+1\over 2}}{{j-3\over 2}}}C_{n-{j-1\over 2}}+(-1)^{{j-1 \over 2}}{\binom{{j-1\over 2}}{{j-1\over 2}}}C_{n-{j+1\over 2}},
\end{equation}
from which the result follows. If $j$ is even, the result can be proven similarly.

\noindent
(b) For each positive integer $n \geq 1$, by Theorem \ref{structure} we have
$$q_{n1}={1\over 2}{\binom{2n}{n}}={{2n(2n-1)}\over n(2n)}{(2n-2)! \over {(n-1)!(n-1)!}}=(2n-1)C_{n-1}.$$
For each $j\geq 2$ with $n \geq j$, we prove the result by induction on $j$.
By (\ref{Lucare}),
$$q_{n2}=(2n-1)C_{n-1}-nC_{n-1}=(n-1)C_{n-1}\quad\mbox{and}\quad q_{n3}=(n-1)C_{n-1}-(2n-3)C_{n-2},$$
which is the result (\ref{Lucarel}) for $j=2$ and $3$ with $n \geq j$. The induction commences:
Assume that $j\geq 4$ is even. 
Then by the induction hypothesis,
\begin{equation}
\label{Lurowsum1}
\begin{split}
q_{n-1,j-2}&=(n-2)C_{n-2}-\left[(n-3)\left({\binom{j-4}{1}}+{\binom{j-5}{0}}\right)+{\binom{j-5}{0}}\right]C_{n-3}\\
&+\left[(n-4)\left({\binom{j-5}{2}}+{\binom{j-6}{1}}\right)+{\binom{j-6}{1}}\right]C_{n-4}+\cdots\\
&+(-1)^{{j-4}\over 2}\left[\left(n-{{j}\over 2}\right)\left({\binom{{{j-2}\over 2}}{{{j-4}\over 2}}}+{\binom{{{j-2}\over 2}-1}{{{j-4}\over 2}-1}}\right)+{\binom{{{j-4}\over 2}}{{{j-6}\over 2}}}\right]C_{n-{{{{j}\over 2}}}},
\end{split}
\end{equation}
\begin{equation}
\label{Lurowsum2}
\begin{split}
q_{n,j-1}&=(n-1)C_{n-1}-\left[(n-2)\left({\binom{j-3}{1}}+{\binom{j-4}{0}}\right)+{\binom{j-4}{0}}\right]C_{n-2}\\
&+\left[(n-3)\left({\binom{j-4}{2}}+{\binom{j-5}{1}}\right)+{\binom{j-5}{1}}\right]C_{n-3}+\cdots\\
&+(-1)^{{{{j-2}\over 2}}}\left[\left(n-{{{j}\over 2}}\right)\left({\binom{{{j-2}\over 2}}{{{j-2}\over 2}}}+{\binom{{{j-2}\over 2}-1}{{{j-2}\over 2}-1}}\right)+{\binom{{{j-4}\over 2}}{{{j-4}\over 2}}}\right]C_{n-{{{{j}\over 2}}}}.
\end{split}
\end{equation}
\newline
Once more, by (\ref{Lucare}), as well as (\ref{Lurowsum1}) and (\ref{Lurowsum2}), it follows that
\begin{equation}
\label{Lurowsum3}
\begin{split}
q_{n,j}&=(n-1)C_{n-1}-\left[(n-2)\left({\binom{j-2}{1}}+{\binom{j-3}{0}}\right)+{\binom{j-3}{0}}\right]C_{n-2}\\
&+\left[(n-3)\left({\binom{j-3}{2}}+{\binom{j-4}{1}}\right)+{\binom{j-4}{1}}\right]C_{n-3}+\cdots\\
&+(-1)^{{{{j-2}\over 2}}}\left[\left(n-{{{j}\over 2}}\right)\left({\binom{{{j}\over 2}}{{{j-2}\over 2}}}+{\binom{{{j}\over 2}-1}{{{j-2}\over 2}-1}}\right)+{\binom{{{j-2}\over 2}}{{{j-4}\over 2}}}\right]C_{n-{{{{j}\over 2}}}},\\
\end{split}
\end{equation}
which is the desired result (\ref{Lucarel}). The case of odd $j$ can be proven similarly.
\end{proof}
Theorem \ref{FiboLucaentry} says that for a pair $n$ and $j$ of positive integers with $n \geq j$, each entry of
$D({\mathbb F}^{\rm S})^{-1}D$ and $D({\mathbb L}^{\rm S})^{-1}D$ can be expressed as a linear combination of Catalan numbers with coefficients
such that
$$r_{nj}={\binom{j-1}{0}}C_{n-1}-{\binom{j-2}{1}}C_{n-2}+\cdots+(-1)^{\lfloor{{{j-1}\over 2}\rfloor}}{\binom{\lceil{{j-1}\over 2}\rceil}{\lfloor{{j-1}\over 2}\rfloor}}C_{n-\lceil{{{j}\over 2}\rceil}},$$
where ${\binom{j-1}{0}}+{\binom{j-2}{1}}+\cdots+{\binom{\lceil{{j-1}\over 2}\rceil}{\lfloor{{j-1}\over 2}\rfloor}}$
is the $j$th Fibonacci number $F_j$ for $j\geq 1$ and
\begin{eqnarray*}
q_{nj}&=&(n-1){\binom{j-1}{0}}C_{n-1}-[(n-2)({\binom{j-2}{1}}+{\binom{j-3}{0}})+{\binom{j-3}{0}}]C_{n-2}
+[(n-3)({\binom{j-3}{2}}\\
&+& {\binom{j-4}{1}})+{\binom{j-4}{1}}]C_{n-3} +  \cdots
+(-1)^{\lfloor{{{j-1}\over 2}\rfloor}}[(n-\lceil{{{j}\over 2}\rceil})({\binom{\lceil{{j-1}\over 2}\rceil}{\lfloor{{j-1}\over 2}\rfloor}}+{\binom{\lceil{{j-1}\over 2}\rceil-1}{\lfloor{{j-1}\over 2}\rfloor-1}})\\
&+&{\binom{\lceil{{j-3}\over 2}\rceil}{\lfloor{{j-3}\over 2}\rfloor}}]C_{n-{\lceil{{{j}\over 2}\rceil}}},
\end{eqnarray*}
where ${\binom{j-1}{0}}+({\binom{j-2}{1}}+{\binom{j-3}{0}})+({\binom{j-3}{2}}+{\binom{j-4}{1}})+\cdots+({\binom{\lceil{{j-1}\over 2}\rceil}{\lfloor{{j-1}\over 2}\rfloor}}+{\binom{\lceil{{j-1}\over 2}\rceil-1}{\lfloor{{j-1}\over 2}\rfloor-1}})$ is the $(j-1)$th Lucas number $L_{j-1}$ and ${\binom{j-3}{0}}+{\binom{j-4}{1}}+\cdots+{\binom{\lceil{{j-3}\over 2}\rceil}{\lfloor{{j-3}\over 2}\rfloor}}$ is the $(j-2)$th Fibonacci number $F_{j-2}$ for $j\geq 3$.

\noindent
For example, it follows from Theorem \ref{FiboLucaentry} that $D({\mathbb F}^{\rm S})^{-1}D$ and $D({\mathbb L}^{\rm S})^{-1}D$ can be expressed
in terms of the Catalan numbers such that each entry of the matrices entails the Fibonacci or Lucas sequences:
\[
      D({\mathbb F}^{\rm S})^{-1}D=\left[\begin{array}{ccccccccc}
           ~ C_0 & 0   & 0   & 0       & 0        &    \cdots~\\
            ~0   & C_0 & 0   & 0       & 0        &    \cdots~\\
            ~0   & C_1 & C_1 & 0       & 0        &    \cdots~\\
            ~0   & C_2 & C_2 & C_2-C_1 & 0        &    \cdots~\\
            ~0   & C_3 & C_3 & C_3-C_2 & C_3-2C_2 &   \cdots~\\
             ~\vdots & \vdots & \vdots & \vdots &\vdots & \ddots~
                   \end{array}\right]
\] and
\[
      D({\mathbb L}^{\rm S})^{-1}D=\left[\begin{array}{ccccccccc}
           ~ C_0  & 0      & 0      & 0        & 0         &     \cdots~\\
            ~2C_1 & C_0    & 0      & 0        & 0         &    \cdots~\\
            ~3C_2 & 3C_1   & C_1    & 0        & 0         &    \cdots~\\
            ~4C_3 & 5C_2   & 2C_2   &2C_2-3C_1 & 0         &     \cdots~\\
            ~5C_4 & 7C_3   & 3C_3   &3C_3-5C_2 &3C_3-7C_2  &    \cdots~\\
              ~\vdots & \vdots & \vdots &\vdots    &\vdots    &  \ddots~
                   \end{array}\right].
\]

\section{The Role of $D({\mathbb F}^{\rm S})^{-1}D$ and $D({\mathbb L}^{\rm S})^{-1}D$ in $P$-invariant sequences}
\par
\setcounter{num}{4}
\setcounter{equation}{0}
In this section, we focus on the role of $D({\mathbb F}^{\rm S})^{-1}D$ and $D({\mathbb L}^{\rm S})^{-1}D$
in transforming $P$-invariant (inverse $P$-invariant) sequences.
This allows us to examine combinatorial relationships between ${\mathbb F}^{\rm S}$ and $D({\mathbb F}^{\rm S})^{-1}D$.
We begin with a known result from \cite{Sun} needed in our subsequent discussion; it can be readily proven by taking advantage of the
Riordan matrix multiplication.
\begin{lem}
\label{Sun}
Let $P$ denote the Pascal matrix. Then
$g(x)$ is the generating function of a $P$-invariant or an inverse $P$-invariant sequence of the first kind if and only if
${1 \over {1-x}}g\left({-x \over {1-x}}\right)=\pm g(x).$
\end{lem}
\begin{proof}
Let $g(x)=\sum_{n=0}^{\infty}a_nx^n$. Then $g(x)$ is the generating function of a $P$-invariant or an inverse $P$-invariant sequence of the first kind if and only if
$({1 \over {1-x}}, {-x \over {1-x}})[a_0, a_1, a_2, \ldots]^T=\pm[a_0, a_1, a_2, \ldots]^T$, which stands for ${1 \over {1-x}}g\left({-x \over {1-x}}\right)=\pm g(x)$
because $PD=({1 \over {1-x}}, {-x \over {1-x}})$ where $D=(1,-x)$.
\end{proof}
The following result follows from Lemmas \ref{columns} and \ref{Sun}.
\begin{cor}
\label{sol10}
Let $P$ denote the Pascal matrix. Then $g(x)$ is the generating function of a $P$-invariant sequence of the first kind if and only if ${{x} \over {2-x}}g(x)$ is the generating function of an inverse $P$-invariant sequence of the first kind.
  \end{cor}
\begin{proof}
Let $(l(x), h(x))$ denote a semi-Riordan matrix such that ${\mathbb F}^{\rm F}=(l(x), h(x)){\mathbb L}^{\rm F}$. Then by Lemma \ref{columns} (c) and (d),
$h(x)$ satisfies $(1-x)h(x)^2+x^2h(x)-x^2=0$. So $h(x)={{-x^2\pm \sqrt{x^4+4(1-x)x^2}}\over 2(1-x)}$, which implies that $h(x)={-x \over {1-x}}$ or $x$.
Thus we have $(l(x), h(x))=({{x} \over {(1-x)(2-x)}}, {-x \over {1-x}})$ or $({{x} \over {2-x}}, x)$. Let ${\mathbf w}$ be a $P$-invariant sequence of the first kind with its generating function $g(x)$. Then ${\mathbf w}={\mathbb L}^{\rm F}{\mathbf u}$ for some ${\mathbf u}  \in  \mathbb{R}^\infty$ and we have
$(l(x),h(x)){\mathbf w}={\mathbb F}^{\rm F}{\mathbf u}$.
So by Lemma \ref{columns} (c), $(l(x),h(x)){\mathbf w}$ is an inverse $P$-invariant sequence of the first kind with its generating function ${{x} \over {2-x}}g(x)$
because ${{x} \over {(1-x)(2-x)}}g({{-x} \over {1-x}})={{x} \over {2-x}}g(x)$ by Lemma \ref{Sun}. The other direction can be proven similarly.
\end{proof}
Next we examine the role of $D({\mathbb F}^{\rm S})^{-1}D$ and $D({\mathbb L}^{\rm S})^{-1}D$ with the Catalan numbers
for converting a $P$-invariant to an inverse $P$-invariant sequence of the first kind, and vice versa.
\begin{thm}
\label{sol11}
For ${\mathbf v}\in  \mathbb{R}^\infty$, let ${\mathbf w}={\mathbb F}^{\rm F}{\mathbf v}$ and ${\mathbf x}={\mathbb L}^{\rm F}{\mathbf v}$. Then the following hold.
\begin{itemize}
  \item [(a)] $\left({{2-xC(x)} \over {xC(x)}}, xC(x)\right){\mathbf w}=D({\mathbb F}^{\rm S})^{-1}D{\mathbf x}$,
  \item [(b)] $\left({{2-{xC(x)}} \over {{xC(x)}-2{x^2C(x)^2}}}, {xC(x)}\right){\mathbf w}=D({\mathbb L}^{\rm S})^{-1}D{\mathbf x}$,
  \item [(c)] $\left({{xC(x)} \over {2-{xC(x)}}},{xC(x)}\right){\mathbf x}=D({\mathbb F}^{\rm S})^{-1}D{\mathbf w}$,
  \item [(d)] $\left({{xC(x)} \over {(1-2{xC(x)})(2-{xC(x)})}}, {xC(x)}\right){\mathbf x}=D({\mathbb L}^{\rm S})^{-1}D{\mathbf w}$ where $C(x)={{1-\sqrt{1-4x}}\over 2x}$.
\end{itemize}
\end{thm}
\begin{proof}
(a) We already know that ${\mathbb F}^{\rm F}=\left({x \over {2-x}}, x \right){\mathbb L}^{\rm F}$ or ${\mathbb F}^{\rm F}=\left({{x} \over {(1-x)(2-x)}}, {{-x} \over {1-x}}\right){\mathbb L}^{\rm F}$ as in the proof of Corollary \ref{sol10}. Since
\begin{equation*}
\begin{split}
\left({{-1+{2 \over {xC(x)}}} \over {xC(x)-1}}, {{xC(x)} \over {xC(x)-1}}\right)\left({{x} \over {(1-x)(2-x)}}, {{-x} \over {1-x}}\right)&=\left({{2} \over {xC(x)}}-1, xC(x)\right)\left({x \over {2-x}}, x \right)\\
&=(1, xC(x)),
\end{split}
\end{equation*}
Clauses (a) and (b) directly follow by Lemmas \ref{inverse}. Clauses (c) and (d) can be proven similarly.
\end{proof}
For ${\mathbf v} \in \mathbb{R}^\infty$, let $GF({\mathbf v})$ denote the generating function of ${\mathbf v}$. In the following corollary, the Catalan and Motzkin numbers play a critical role in transforming the generating functions of $P$-invariant and inverse $P$-invariant sequences of the first kind.
\begin{cor}
\label{sol12}
For ${\mathbf v}  \in  \mathbb{R}^\infty$, let $g_{-{\mathbf v}}(x)$ resp. $g_{\mathbf v}(x)$ denote $GF({\mathbb F}^{\rm F}{\mathbf v})$ resp. $GF({\mathbb L}^{\rm F}{\mathbf v})$. Then we have the following:
\begin{itemize}
  \item [(a)] $g_{\mathbf v}(xC(x))=\left(C(x)- {2 \over x}\right)g_{-{\mathbf v}}(x-C(x))$,
  \item [(b)] $g_{\mathbf v}(x-C(x))=\left(2xC(x)^2+xC(x)-4C(x)+ {{2-x} \over x}\right)g_{-{\mathbf v}}(xC(x))$,
  \item [(c)] $g_{\mathbf v}(xC(x))=\left({x-2 \over x}+{x \over {1-x}}M({x \over {1-x}})\right)g_{-{\mathbf v}}(x-1-{x \over {1-x}}M({x \over {1-x}}))$,
  \item [(d)] $g_{\mathbf v}(x-1-{x \over {1-x}}M({x \over {1-x}}))=\left((3x-4)C(x)+{{2x^2C(x)} \over {1-x}}M({x \over {1-x}})+ {{2-x} \over x}\right)g_{-{\mathbf v}}(xC(x))$, where $C(x)={{1-\sqrt{1-4x}}\over 2x}$ and $M(x)={{1-x- \sqrt{(1-x)^2-4x^2}}\over 2x^2}$.
\end{itemize}
\end{cor}
\begin{proof}
By Lemma \ref{Sun} and Theorem \ref{sol11} (a) and (c), we have
$g_{\mathbf v}(t)={{2-t} \over t}g_{-\mathbf v}(t)=-{{2-t} \over {t(1-t)}}g_{-\mathbf v}({{-t} \over {1-t}})$
and $g_{-\mathbf v}(t)={{t} \over {2-t}}g_{\mathbf v}(t)={{t} \over {(2-t)(1-t)}}g_{\mathbf v}({{-t} \over {1-t}})$ where $t=xC(x)$.
Thus (a) and (b) follow by the direct calculations.
By using the Euler transformation in the proof of Lemma \ref{inverse}, we obtain
\begin{equation}
\label{Euler}
C(x)=1+{x \over {1-x}}M({x \over {1-x}}).
\end{equation}
Applying (\ref{Euler}) to (a) and (b) yields (c) and (d).
\end{proof}

In the next theorem, we present one of our main results, namely an answer to {\bf Q2}, featuring a combinatorial relationship between ${\mathbb F}^{\rm S}$ and $D({\mathbb F}^{\rm S})^{-1}D$ and furthermore, between ${\mathbb L}^{\rm S}$ and $D({\mathbb L}^{\rm S})^{-1}D$. The result reveals
that $D({\mathbb F}^{\rm S})^{-1}D$ and $D({\mathbb F}^{\rm S})^{-1}D$ provide a mechanism (that features the Catalan numbers) for transforming
a $P$-invariant into an inverse $P$-invariant sequence  of the second kind, and vice versa.
\begin{thm}
\label{sol13}
For ${\mathbf v}  \in  \mathbb{R}^\infty$, let ${\mathbf y}={\mathbb F}^{\rm S}{\mathbf v}$, ${\mathbf z}={\mathbb L}^{\rm S}{\mathbf v}$. Then the following hold:
\begin{itemize}
  \item [(a)] $\left({{1} \over {1+2xC(x)}}, xC(x)\right){\mathbf z}=D({\mathbb F}^{\rm S})^{-1}D{\mathbf y}$,
  \item [(b)] $\left({{1} \over {1-4x^2C(x)^2}}, xC(x)\right){\mathbf z}=D({\mathbb L}^{\rm S})^{-1}D{\mathbf y}$,
  \item [(c)] $\left(1+2xC(x), xC(x)\right){\mathbf y}=D({\mathbb F}^{\rm S})^{-1}D{\mathbf z}$,
  \item [(d)] $\left({{1+2xC(x)} \over {1-2xC(x)}}, xC(x)\right){\mathbf y}=D({\mathbb L}^{\rm S})^{-1}D{\mathbf z}$ where $C(x)={{1-\sqrt{1-4x}}\over 2x}$.
\end{itemize}
\end{thm}
\begin{proof}
(a) From Lemma \ref{columns} (a) and (b), it can be derived easily that ${\mathbb L}^{\rm S}=\left(1+2x, x \right){\mathbb F}^{\rm S}$. Since $\left({{1} \over {1+2xC(x)}}, xC(x)\right)\left(1+2x, x \right)=(1, xC(x)),$ (a) follows directly from Lemma \ref{inverse} (a). Clauses (b), (c), and (d) can be proven similarly.
\end{proof}
We conclude this section with the relationships (involving the Catalan and Motzkin numbers) between the generating functions
of $P$-invariant and inverse $P$-invariant sequences of the second kind.
\begin{cor}
\label{sol14}
For ${\mathbf v}  \in  \mathbb{R}^\infty$, let $h_{\mathbf v}(x)$ resp. $h_{-\mathbf v}(x)$ denote $GF({\mathbb F}^{\rm S}{\mathbf v})$ resp. $GF({\mathbb L}^{\rm S}{\mathbf v})$. Then we have the following:
\begin{itemize}
  \item [(a)] $h_{\mathbf v}(xC(x))={{1} \over {1+2xC(x)}}h_{-{\mathbf v}}(xC(x))$,
  \item [(b)] $h_{-{\mathbf v}}(x+{x^2 \over {1-x}}M({x \over {1-x}}))=\left(1+2x+{{2x^2} \over {1-x}}M({x \over{1-x}})\right) h_{\mathbf v}(x+{x^2 \over {1-x}}M({x \over {1-x}}))$
  where
  $C(x)={{1-\sqrt{1-4x}}\over 2x}$ and $M(x)={{1-x- \sqrt{(1-x)^2-4x^2}}\over 2x^2}$.
 \end{itemize}
\end{cor}
\begin{proof}
Clause (a) follows from Lemma \ref{inverse} and Theorem \ref{sol13} and (b) is due to (\ref{Euler}).
\end{proof}

\section{A connection between Riordan Involutions  and Invariant Sequences}
\par
\setcounter{num}{5}
\setcounter{equation}{0}
In this section, the Riordan (pseudo) involution $R$ is applied to construct $R$-invariant sequences of the first or second kind, as in Definition \ref{Rinvseq}.
From this, we show that for every Riordan involution $R$ in the $(-1)$-Appell subgroup, there exists $B^n(n=1, 2, \ldots)$ in the Appel subgroup such that $R=B^nDB^{-n}$.
Interestingly, for each $n=1, 2, \ldots$ such a $B^n$ can be directly constructed from $R$ itself and has a related combinatorial significance for
$R$-invariant or inverse $R$-invariant sequences of the first kind.

The following lemma provides a method for constructing $R$-invariant sequences of the first or second kind by means of the Riordan (pseudo) involution $R$ itself.
\begin{lem}
\label{Rinv}
Let $U=RD$ and $V=DR$, where $D=(1, -x)$.
For each positive integer $n$, let $R$ be a Riordan (pseudo) involution. Then the following hold:
\begin{itemize}
  \item [(a)] the columns of $(U+D)^n~((R+D)^n)$ are $R$-invariant sequences of the first kind.
  \item [(b)] the columns of $(U-D)^n ~ ((R-D)^n)$ are inverse $R$-invariant sequences of the first kind.
  \item [(c)] the columns of $(V^T+D)^n~((R^T+D)^n)$ are $R$-invariant sequences of the second kind.
  \item [(d)] the columns of $(V^T-D)^n~((R^T-D)^n)$ are inverse $R$-invariant sequences of the second kind.
\end{itemize}
\end{lem}
\begin{proof}
Let $U=RD$ and $V=DR$ where $R$ is a Riordan involution and $D=(1, -x)$.
Then for each positive integer $n$, since $R$ is a Riordan involution, we have
$R(U+D)^n=(RRD+U)(U+D)^{n-1}=(U+D)^n$
resp. $R^T(V^T-D)^n=(R^TR^TD-V^T)(V^T-D)^{n-1}=-(V^T-D)^n$, which imply that each column of $(U+D)^n$ is a $R$-invariant sequence of the first kind resp. each column of $(V^T-D)^n$ is an inverse $R$-invariant sequence of the second kind. The two clauses (b) and (c) are also proven similarly.
\end{proof}
For the $(-1)$-Appell subgroup, Lemma \ref{Rinv} allows us to answer {\bf Q8}. First, we need a simple lemma, in which $I$
denotes as usual the infinite identity matrix.
\begin{lem}
\label{Radd}
Let $R=(g(x), f(x))$ be a Riordan matrix with positive main diagonal entries. Then $R+I$ is a Riordan matrix if and only if $f(x)=x$.
\end{lem}
\begin{proof}
Let $R=(g(x), f(x))$ be a Riordan matrix with positive main diagonal entries. If $f(x)=x$, then for $i=0, 1, \ldots$, the generating function of the $i$th column of $R+I$ is $g(x)x^i+x^i=(g(x)+1)x^i$. So $R+I$ is a Riordan matrix with $R+I=(g(x)+1, x)$. Conversely, assume that $R+I=(h(x), l(x))$ is a Riordan matrix. Then clearly, we have $h(x)=g(x)+1$ and for $i=0, 1, \ldots$,
the $i$th column of $R+I$ is $(g(x)+1)l(x)^i=g(x)f(x)^i+x^i$, from which we get ${{(g(x)f(x)+x)^2} \over {(g(x)+1)^2}}={{g(x)f(x)^2+x^2} \over {g(x)+1}}$, and $2xf(x)=f(x)^2+x^2$. Thus $f(x)=x$.
\end{proof}

The following theorem is an affirmative answer to {\bf Q8} for every Riordan involution $R$ in the $(-1)$-Appell subgroup. In particular, for each positive integer $n=1, 2, \ldots$, the matrix $B^n$, which is a Riordan matrix in the Appell subgroup and satisfies the condition $R=B^nDB^{-n}$, can be directly obtained from the Riordan involution $R$.
\begin{thm}
\label{ShapiroRinv}
Let $R=(g(x), -x)$ be a Riordan involution such that the diagonal entries of $RD$ are positive, where $D=(1, -x)$. Then there exists a Riordan matrix $B^n$ such that
$R=B^nDB^{-n}$, where $B^n=((g(x)+1)^n, x)$ for each $n=1, 2, \ldots$
\end{thm}
\begin{proof}
If $R=(g(x), -x)$ is a Riordan involution, then it follows from Lemma \ref{Rinv} that for each $n=1, 2, \ldots$, we have $R(U+D)^n=(U+D)^n$ and $R(U-D)^n=-(U-D)^n$,
where $U=RD$. Let $(U+D)^n=[{\mathbf x}_0^n, {\mathbf x}_1^n, {\mathbf x}_2^n, \ldots]$ and
$(U-D)^n=[{\mathbf y}_0^n, {\mathbf y}_1^n, {\mathbf y}_2^n, \ldots]$. Then for each $n=1, 2, \ldots$, $$R[{\mathbf x}_0^n, {\mathbf y}_1^n, {\mathbf x}_2^n, {\mathbf y}_3^n, \ldots]=[{\mathbf x}_0^n, -{\mathbf y}_1^n, {\mathbf x}_2^n, -{\mathbf y}_3^n, \ldots]=[{\mathbf x}_0^n, {\mathbf y}_1^n, {\mathbf x}_2^n, {\mathbf y}_3^n, \ldots]D,$$
which can be represented by $RB^n=B^nD$ where $B^n=[{\mathbf x}_0^n, {\mathbf y}_1^n, {\mathbf x}_2^n, {\mathbf y}_3^n, \ldots]$.
In fact, it follows from Lemma \ref{Radd} that $U+I$ is a Riordan matrix with $U+I=(g(x)+1, x)$. Hence for each $n=1, 2, \ldots$, there exists a Riordan matrix $B^n$ such that
$R=B^nDB^{-n}$ where $B^n=(U+I)^n=((g(x)+1)^n, x)$, and the proof is complete.
\end{proof}

Claimed next is an answer to {\bf Q8.1}, namely, that for a Riordan involution $R$ in the $(-1)$-Appell subgroup,
$B^n (n=1, 2, \ldots)$ has a related combinatorial significance.
\begin{thm}
\label{ShapiroRinv1}
Let $R=(g(x), -x)$ be a Riordan involution such that the diagonal entries of $RD$ are positive where $D=(1, -x)$. Then there exists a Riordan matrix $B^n$
such that the columns of $B^n$ for $n=1, 2, \ldots$ are $R$-invariant or inverse $R$-invariant sequences of the first kind.
\end{thm}
\begin{proof}
It directly follows from Lemma \ref{Rinv} and Theorem \ref{ShapiroRinv}.
\end{proof}
\begin{rmk}
\label{ShiftedR}
{\rm
For each Riordan matrix $R$, let us refer to $R+I$ as the shifted Riordan matrix of $R$ by $I$. In Shapiro's open questions {\bf Q8} and {\bf Q8.1}, if the condition on
$B$ is replaced by integral powers of a shifted Riordan matrix by $I$, one gets more meaningful answers by applying Theorem \ref{ShapiroRinv}. This is
evident in the following theorem. }
\end{rmk}
\begin{thm}
\label{ShapiroRinv2}
Let $R$ be a Riordan (resp., pseudo) involution and let $B$ be a shifted Riordan matrix of $RD$ (resp., $R$ ) by $I$. Then for each $n=1, 2, \ldots$,
$R=B^nDB^{-n}$ (resp., $RD=B^nDB^{-n}$) such that the columns of $B^n$ are $R$-invariant or inverse $R$-invariant sequences of the first kind.
\end{thm}

To illustrate the above theorem,
let $R=((1+xC(x))C(x), x(1+xC(x))C(x))$, where $C(x)={{1-\sqrt{1-4x}}\over 2x}$. Then $R$ is a pseudo involution \cite{Cheon1} with
\[
      R=\left({{1-x-\sqrt{1-4x}}\over x}, 1-x-\sqrt{1-4x}\right)=\left[\begin{array}{ccccccccc}
           ~ 1 & 0 & 0 & 0 & 0  & 0 &  \cdots~\\
            ~2 & 1 & 0 & 0 & 0  & 0 & \cdots~\\
            ~4 & 4 & 1 & 0 & 0  & 0 & \cdots~\\
           ~10 &12 & 6 & 1 & 0  & 0 &  \cdots~\\
           ~28 &36 &24 & 8 & 1  & 0 & \cdots~\\
           ~84 &112 &96 &40 & 10  & 1 & \cdots~\\
       ~\vdots & \vdots & \vdots & \vdots &\vdots & \vdots& \ddots~
                   \end{array}\right].
\]
From the previous results, it can be easily derived that for each $n=1, 2, \ldots$, $RD=B^nD{B^{-n}}$ and the columns
of $B^n$ are $R$-invariant or inverse $R$-invariant sequences of the first kind, where $B=R+I$.
\vskip 0.5cm
\section*{References}
\bibliography{mybibfile}

\end{document}